\documentclass{article}
\usepackage[utf8]{inputenc}
\usepackage{amsthm}
\usepackage{amsmath}
\usepackage{amssymb}
\usepackage{amsfonts}
\usepackage{mathrsfs}
\usepackage{xcolor}
\usepackage{enumerate}
\usepackage{graphicx}
\usepackage{tikz}
\usetikzlibrary{calc}
\usepackage{multicol}
\usepackage{sansmath}
\usepackage{microtype}
\usepackage{xurl}
\usepackage{authblk}

\newtheorem{thm}{Theorem}[section]
\newtheorem{conj}[thm]{Conjecture}
\newtheorem{cor}[thm]{Corollary}

\newtheorem{lemma}[thm]{Lemma}
\newtheorem{prop}[thm]{Proposition}

\theoremstyle{definition}
\newtheorem{ex}[thm]{Example}

\newcommand{\cb}{c_2}
\newcommand{\pc}{p}
\DeclareMathOperator{\mr}{mr}
\DeclareMathOperator{\rank}{rank}
\DeclareMathOperator{\cc}{cc}

\title{Subgraph complementation and minimum rank}

\author[1]{Calum Buchanan\thanks{Calum.Buchanan@uvm.edu}}
\author[2]{Christopher Purcell}
\author[1]{Puck Rombach}

\affil[1]{\footnotesize{Department of Mathematics \& Statistics, University of Vermont}}
\affil[2]{\footnotesize{Department of Mathematics, University of West Bohemia}}

\date{\today}

\begin{document}

\maketitle

\begin{abstract}
Any finite simple graph $G = (V,E)$ can be represented by a collection $\mathscr{C}$ of subsets of $V$ such that $uv\in E$ if and only if $u$ and $v$ appear together in an odd number of sets in $\mathscr{C}$. Let $\cb(G)$ denote the minimum cardinality of such a collection. This invariant is equivalent to the minimum dimension of a faithful orthogonal representation of $G$ over $\mathbb{F}_2$ and is closely connected to the minimum rank of $G$. We show that $\cb (G) = \mr(G,\mathbb{F}_2)$ when $\mr(G,\mathbb{F}_2)$ is odd, or when $G$ is a forest. Otherwise, $\mr(G,\mathbb{F}_2)\leq \cb (G)\leq \mr(G,\mathbb{F}_2)+1$. Furthermore, we show that the following are equivalent for any graph $G$ with at least one edge: {\em i.} $\cb(G)=\mr(G,\mathbb{F}_2)+1$; {\em ii.} the adjacency matrix of $G$ is the unique matrix of rank $\mr(G,\mathbb{F}_2)$ which fits $G$ over $\mathbb{F}_2$; {\em iii.} there is a minimum collection $\mathscr{C}$ as described in which every vertex appears an even number of times; and {\em iv.} for every component $G'$ of $G$, $\cb(G') = \mr(G',\mathbb{F}_2) + 1$. We also show that, for these graphs, $\mr(G,\mathbb{F}_2)$ is twice the minimum number of tricliques whose symmetric difference of edge sets is $E$. Additionally, we provide a set of upper bounds on $\cb(G)$ in terms of the order, size, and vertex cover number of $G$. Finally, we show that the class of graphs with $\cb(G)\leq k$ is hereditary and finitely defined. For odd $k$, the sets of minimal forbidden induced subgraphs are the same as those for the property $\mr(G,\mathbb{F}_2)\leq k$, and we exhibit this set for $\cb(G)\leq2$.
\end{abstract}

\section{Introduction}

Given any two finite simple graphs $G$ and $H$ on a set $V$ of $n$ vertices, one can obtain $G$ from $H$ by a sequence of {\em subgraph complementations}, the operation of complementing the edge set of an induced subgraph. That is, there exist graphs $H_0, H_1, \ldots, H_k$ such that $H_0 = H$, $H_k = G$, and $H_{i}$ is obtainable from $H_{i-1}$ by a subgraph complementation for each $i = 1, 2, \ldots, k$. Trivially, one can complement each edge of $G$ which is not an edge of $H$ and each non-edge of $G$ which is an edge of $H$. It is natural to ask for the minimum number of subgraph complementations needed to obtain $G$ from $H$, which we call the {\em subgraph complementation distance} between $G$ and $H$. This problem is equivalent to that of finding the subgraph complementation distance between the empty graph, $\overline{K}_n$, and the {\em symmetric difference} of $G$ and $H$, the graph on $V$ whose edges appear in exactly one of $G$ or $H$. Thus, we are particularly interested in the subgraph complementation distance between $G$ and $\overline{K}_n$, which we call the {\em subgraph complementation number of $G$} and denote by $\cb(G)$.
\par

The operation of subgraph complementation was defined by Kami{\'n}ski, Lozin, and Milani{\v{c}}~\cite{kaminski2009recent} in the study of graphs with bounded clique-width. Variations of the operation appeared earlier, such as complementation of the subgraph induced by the open neighborhood of a vertex, called {\em local complementation}~\cite{bouchet1994circle}. For a graph class $\mathscr{G}$ and graph $G$, the problem of determining whether some subgraph complementation of $G$ results in a graph in $\mathscr{G}$ is studied in~\cite{fomin2020subgraph}.
\par

We call a collection $\mathscr{C}$ of subsets of $V$ with respect to which successive subgraph complementations of $\overline{K}_n$ result in $G$ a {\em subgraph complementation system} for $G$. Equivalently, $\mathscr{C}$ is a subgraph complementation system for $G$ if each pair of adjacent vertices in $G$ is contained in an odd number of sets in $\mathscr{C}$ and each pair of non-adjacent vertices in an even number.
Multiple problems have been posed which are equivalent to finding subgraph complementation systems or to finding $c_2(G)$. Vatter asked for ways to express the edge set of $G$ as a sum modulo 2 of edge sets of cliques~\cite{317716}; a subgraph complementation system for $G$ may be interpreted as a collection of complete graphs on subsets of $V$ whose symmetric difference of edge sets is $E(G)$, and $\cb(G)$ is the minimum cardinality of such a collection. An {\em orthogonal representation} of $G$ over a field $\mathbb{F}$ is an assignment of vectors from $\mathbb{F}^d$ to the vertices of $G$ such that nonadjacent vertices are represented by orthogonal vectors. Lov\'asz introduced orthogonal representations over $\mathbb{R}$ to bound the Shannon capacity of a graph~\cite{lovasz1979shannon}. Alekseev and Lozin examined the minimum dimension of an orthogonal representation in which adjacent vertices are represented by vectors whose dot product is 1~\cite{alekseev2001orthogonal}.\footnote{This is sometimes called an {\em exact dot product representation}~\cite{tucker2007exact},\cite{minton2008dot}.} When the field in question is $\mathbb{F}_2$, the field of order $2$, this is equivalent to the problem of finding $c_2(G)$.
\par

An orthogonal representation of $G$ over $\mathbb{F}$ is called {\em faithful} if adjacent vertices are represented by nonorthogonal vectors. When $\mathbb{F}=\mathbb{F}_2$, these are the representations studied in~\cite{alekseev2001orthogonal}. A faithful orthogonal representation of $G$ over $\mathbb{F}_2$ of dimension $d$ induces a subgraph complementation system $\mathscr{C} = \{C_1, C_2, \ldots, C_d\}$ of $G$ by including a vertex $v$ in $C_i$ if and only if the $i$th entry of the vector associated to $v$ is 1. Similarly, given a subgraph complementation system $\mathscr{C}$ for $G$, we may assign to each $v\in V$ a vector from $\mathbb{F}_2^d$ with entry $i$ equal to 1 if $v\in C_i$, and $0$ otherwise. The problem of minimizing the dimension of a faithful orthogonal representation over $\mathbb{R}$ is addressed in \cite{lovasz1989orthogonal}. These representations have been generalized in many ways, one of which we have seen in the previous paragraph. In the most general case, we have {\em vector representations} of $G$, introduced by Parsons and Pisanski in~\cite{parsons1989vector}.
\par

Given a graph $G$ and a faithful orthogonal representation of $G$ over $\mathbb{F}_2$, consider the $n\times k$ matrix $M$ with rows given by the vectors in the representation (when it is helpful to specify the corresponding subgraph complementation system, we write $M = M(\mathscr{C})$ or say $M$ is associated to $\mathscr{C}$). An off-diagonal entry of the $n\times n$ matrix $A = MM^T \pmod{2}$ is 0 if and only if the corresponding vertices are nonadjacent. That is, the off-diagonal zeros of $A$ are precisely those of the adjacency matrix of $G$. A matrix with this property is said to {\em fit} $G$. It is a well-studied problem to determine the minimum rank of a matrix which fits $G$ over a given field $\mathbb{F}$. In particular, the minimum rank of $G$ over $\mathbb{R}$ has been of interest for its equivalence to the determination of the maximum multiplicity of an eigenvalue among the family of matrices which fit $G$~\cite{fallat2007minimum}. We denote by $\mr(G,\mathbb{F})$ the minimum rank of a symmetric matrix over $\mathbb{F}$ which fits $G$.
\par

It is not hard to see that the rank of a matrix $M$ over a field $\mathbb{F}$ is at least the rank of $MM^T$. Thus, we see that the dimension of a faithful orthogonal representation of $G$ over $\mathbb{F}$ bounds $\mr(G,\mathbb{F})$ above. In turn, we obtain the bound
\begin{equation}\label{eq:mrlower}
    \mr (G, \mathbb{F}_2) \leq \cb(G).
\end{equation}
In Corollary~\ref{cor:mrcb}, we shall see that, while this bound is not always achieved, $\cb(G)$ and $\mr(G,\mathbb{F}_2)$ differ by no more than 1. Furthermore, we will characterize the graphs $G$ with $\cb(G) = \mr(G,\mathbb{F}_2)+1$ as those whose adjacency matrix is the unique matrix of minimum rank over $\mathbb{F}_2$ which fits $G$.
\par

It is well known that $\mr(G,\mathbb{R})$ is bounded above by the {\em clique covering number} of $G$, $\cc(G)$, or the minimum cardinality of a collection of cliques in $G$ such that every edge of $G$ is in at least one clique~\cite{fallat2007minimum}. Moreover, if every pairwise intersection in a minimal clique covering of $G$ contains at most one vertex, then $\mr(G,\mathbb{F})\leq \cc(G)$ for any field $\mathbb{F}$~\cite{work2008zero}. On the other hand, $\cb(G)$ does not provide a bound for $\mr(G,\mathbb{R})$, significantly differentiating subgraph complementation systems from clique coverings.
\par

The rest of this paper is outlined as follows. In Section~\ref{sec:defnot}, we establish some basic definitions and notation that we will use throughout the paper. In Section~\ref{sec:ubounds}, we elaborate on orthogonal representations of graphs and exhibit a set of upper bounds on $\cb(G)$ for general graphs in terms of their order, size, and vertex cover numbers. In Section~\ref{sec:minrank}, we explore the relationships between $\cb(G)$, $\mr(G,\mathbb{F}_2)$, $\mr(G,\mathbb{R})$, and a new operation: {\em tripartite subgraph complementation}. In Section~\ref{sec:forb}, we show that the graph property $\cb (G)\leq k$ is hereditary and finitely defined, similarly to $\mr (G,\mathbb{F})$ when $\mathbb{F}$ is finite. When $k$ is odd, we show that the sets of forbidden induced subgraphs for $\cb(G)\leq k$ and $\mr(G,\mathbb{F}_2)\leq k$ are the same. We find the minimal forbidden induced subgraphs for the property $\cb (G)\leq 2$.

\section{Definitions and notation}\label{sec:defnot}
All graphs considered in this paper are finite and simple. The vertex set of a graph $G$ is denoted by $V(G)$ and the edge set by $E(G)$, or by $V$ and $E$ respectively when $G$ is evident from context. We denote the number of vertices of $G$ by $|G|$ and the number of edges by $||G||$, or by $n$ and $m$ respectively when $G$ is evident from context. Complete graphs are denoted by $K_n$, paths are denoted by $P_n$, cycles by $C_n$, and wheel graphs by $W_n$, where $n$ indicates the number of vertices in each case. The empty graph $\overline{K}_n$ is the graph complement of $K_n$, and $G$ is called {\em nonempty} if it has at least one edge. The disjoint union of graphs $G$ and $H$ is denoted by $G+H$, and the disjoint union of $k$ copies of $G$ is denoted by $kG$. If $G$ and $H$ are graphs on the same vertex set $V$, the {\em symmetric difference} of $G$ and $H$ is the graph $G\triangle H = (V,E(G)\triangle E(H))$, {\em i.e.} whose edges appear in exactly one of $E(G)$ or $E(H)$. We generalize this definition by considering symmetric differences of graphs $G$ and $H$ on subsets of $V$, defined in the same way. We denote by $N(v)$ the {\em open neighborhood} of a vertex $v$, that is, $N(v) = \{u \mid uv\in E\}$, and by $N[v]$ the {\em closed neighborhood} of $v$, that is, $N(v)\cup \{v\}$. The {\em degree} of $v$ is $|N(v)|$, denoted by $d(v)$. When it is helpful to specify the graph in question, we use the notations $N_G(v)$, $N_G[v]$, and $d_G(v)$, respectively. The induced subgraph of $G$ on the subset of vertices $V\setminus S$ is denoted by $G - S$, and the graph obtained by deleting a vertex $v$ or an edge $e$ is denoted by $G-v$ or $G-e$ respectively. A {\em class} of graphs is a set of graphs closed under isomorphism. A class that is closed under deleting vertices is said to be {\em hereditary}. It is easy to see that a class $X$ is hereditary if and only if there is a set of graphs $M$ such that no graph in $X$ has an induced subgraph in $M$; that is, $X$ may be characterized by its set of {\em minimal forbidden induced subgraphs}.

\section{Orthogonal representations and upper bounds}\label{sec:ubounds}

Alekseev and Lozin studied the minimum dimension of an orthogonal representation of a graph $G$ over a field $\mathbb{F}$ in which the dot product of two vectors representing adjacent vertices is 1~\cite{alekseev2001orthogonal}. In the case that $\mathbb{F}=\mathbb{F}_2$, this is a faithful orthogonal representation of $G$. In keeping with their notation, we let $d(G,\mathbb{F})$ denote the minimum dimension of an orthogonal representation of $G$ over $\mathbb{F}$ such that, for $i\neq j$, $x_i\cdot x_j = 1$ if and only if $ij\in E$. We note that $\cb(G)=d(G,\mathbb{F}_2)$.
\par

We present several upper bounds on the number $\cb(G)$, one in terms of the number of vertices $|G|=n$, one in terms of the number of edges $||G||=m$, and one in terms of the size of a minimum vertex cover $\tau(G)$. Those in terms of $n$ are quoted from \cite{alekseev2001orthogonal}.
\par

\begin{thm}\label{thm:easyvxupp}{\rm \cite{alekseev2001orthogonal}}
For any field $\mathbb{F}$ and any graph $G$ with $n$ vertices,
\[d(G,\mathbb{F})\leq n-1.\]
\end{thm}

\begin{thm}\label{thm:vxuppbound}{\rm \cite{alekseev2001orthogonal}}
For any field $\mathbb{F}$ of characteristic 2 and any $n$-vertex graph $G$ ($n>2$) other than $P_n$,
\[d(G,\mathbb{F}) \leq n-2.\]
Furthermore, $d(P_n, \mathbb{F}) = n-1$.
\end{thm}

It follows that $\cb(G) \leq n-1$ for all graphs $G$, and that equality holds only in the case that $G$ is a path on $n$ vertices. Similarly, as we will see in Proposition~\ref{prop:linearforest}, if $G$ is a {\em linear forest}, or a graph for which every component is a path, then $\cb(G) = m$; otherwise, $\cb(G) \leq m-1$.
\par

\begin{thm}\label{thm:edgeuppbound}
For any graph $G$ with $m$ edges which is not a linear forest,  \[\cb(G)\leq m-1.\]
\end{thm}

\begin{proof}
Suppose that a graph $G$ which is not a linear forest has a vertex $v$ of degree $d(v)>2$. The collection $\{N(v), N[v]\}$ is a subgraph complementation system for the induced subgraph $G[N[v]]$. The remaining $m-d(v)$ edges of $G$ may then be added one at a time to obtain a subgraph complementation system for $G$ of cardinality $m-d(v)+2 \leq m-1$.
\par

Otherwise, $G$ has maximum degree 2. Then $G$ consists of disjoint cycles and paths. Since $G$ is not a linear forest by assumption, it must contain a cycle. Theorem~\ref{thm:vxuppbound} completes the proof.
\end{proof}

\begin{ex}
From Theorem~\ref{thm:vxuppbound} we can deduce the exact subgraph complementation number for cycles. If we were to have $c_2(C_n) \leq n-3$, then we would have $c_2(P_n)\leq n-2$, since $C_n$ and $P_n$ differ in exactly one edge. Therefore, $\cb(C_n)=n-2$.
\par

We can also see that $\cb(C_n)\leq n-2$ by induction. Suppose that $n>3$, and let $v\in V(C_n)$. A subgraph complementation of $C_n$ with respect to $N[v]$ results in an $(n-1)$-cycle and an isolated vertex, which has subgraph complementation number at most $n-3$ by the inductive hypothesis. If $\mathscr{C}$ is a minimum subgraph complementation system for $C_{n-1}$, then $\mathscr{C}\cup \{N[v]\}$ is a subgraph complementation system for $C_n$ of cardinality at most $n-2$.
\end{ex}

We let $\tau(G)$ denote the minimum cardinality of a {\em vertex cover} of $G$, or a set of vertices such that every edge of $G$ is incident to at least one vertex in the set.
\par

\begin{thm}\label{thm:vxcoverupp}
For any graph $G$, \[\cb(G)\leq 2\tau(G).\]
\end{thm}

\begin{proof}
Let $U = \{u_1, \ldots, u_\tau\}\subset V$ be a minimum vertex cover of $G$. Successive subgraph complementations of $\overline{K}_n$ on the sets $N(u_1)$ and $N[u_1]$ yeild each edge incident to $u_1$ in $G$. Some of the edges incident to $u_1$ may also be incident to $u_2$. Thus, in order to obtain the remaining edges incident to $u_2$, we subgraph complement with respect to the sets $N(u_2)\setminus \{u_1\}$ and $N[u_2]\setminus \{u_1\}$. For each $u_i \in U$, we subgraph complement with respect to $N[u_i]\setminus \{u_1, \ldots, u_{i-1}\}$ and $N(u_i)\setminus \{u_1, \ldots, u_{i-1}\}$ to obtain the edges incident to $u_i$ which have not already been built. Since every edge of $G$ is incident to some vertex in $U$ by definition of a vertex cover, and since at most two sets were needed to obtain the edges incident to each vertex in the cover, we have $\cb(G)\leq 2\tau(G)$.
\end{proof}

We remark that $\cb(G)<2\tau(G)$ if any of the sets $N(u_i) \setminus \{u_1,\ldots,u_{i-1}\}$ ($1\leq i\leq \tau$) in the proof of Theorem~\ref{thm:vxcoverupp} are singletons. By reordering, we see that the inequality is strict if there is a minimum vertex cover $U$ of $G$ containing a vertex with only one neighbor outside of $U$.
\par

\section{Minimum rank}\label{sec:minrank}
The problem of finding the minimum cardinality of a subgraph complementation system of a graph relates closely to the minimum rank problem over $\mathbb{F}_2$, and in some cases to the minimum rank problem over $\mathbb{R}$ and other fields. This section explores the nature of these relationships. Unless otherwise specified, when we discuss the rank of a matrix in this section, we mean the rank over $\mathbb{F}_2$.
\par

\subsection{General graphs}\label{sec:mrgeneralgraphs}
We begin by examining the relationship between $\cb(G)$ and $\mr (G, \mathbb{F}_2)$ for general graphs. In Corollary~\ref{cor:mrcb}, we show that $\mr(G,\mathbb{F}_2)\leq \cb(G)\leq \mr(G,\mathbb{F}_2)+1$. In Theorem~\ref{thm:mrneqcbsummary}, we provide a characterization of the graphs for which $\cb(G)=\mr(G,\mathbb{F}_2)+1$. 
\par

\begin{ex}\label{ex:cbneqmr}
Recall, from inequality~\eqref{eq:mrlower} of the introduction, that the subgraph complementation number of a graph $G$ is at least its minimum rank over $\mathbb{F}_2$. This inequality is sharp; we will see in Theorem~\ref{cor:mrcbtree} that equality holds for forests. On the other hand, there are graphs which do not attain equality in \eqref{eq:mrlower}. Consider $K_{3,3}$, the complete bipartite graph with partite sets of order 3. The adjacency matrix of $K_{3,3}$ has only two distinct rows, which are linearly independent over any field, implying that $\mr(K_{3,3}, \mathbb{F}_2) = 2$. However, $c_2(K_{3,3}) > 2$, as we will show in Theorem~\ref{thm:forbind2}; if $A$ and $B$ are the partite sets of $K_{3,3}$, then $\mathscr{C} = \{A,B,A\cup B\}$ is a subgraph complementation system of minimum cardinality.
\end{ex}

It is natural to ask how much larger the subgraph complementation number of a graph might be than its minimum rank over $\mathbb{F}_2$. We will show that, in general, $\cb(G)\leq \mr(G,\mathbb{F}_2)+1$ and detail the cases in which $\cb(G)=\mr(G,\mathbb{F}_2)+1$. The following example, along with Lemmas~\ref{lem:evensymdif} and~\ref{lem:fried}, will be useful. It will also be of use to consider the number of sets in a subgraph complementation system $\mathscr{C}$ of $G$ which contain a given vertex $v$, which we refer to as the number of times that $v$ {\em appears} in $\mathscr{C}$.
\par

\begin{ex}
It is well known that the minimum rank of a graph $G$ over a field $\mathbb{F}$ is additive in the sense that, if $G$ has components $G_1, G_2, \ldots, G_t$, then $\mr(G,\mathbb{F}) = \sum_1^t \mr(G_i,\mathbb{F})$~\cite{fallat2007minimum}. Perhaps surprisingly, the subgraph complementation number behaves differently. The smallest counterexample is given by the graph $G = W_5+K_2$. Trivially, $\cb(K_2) = 1$, and Figure~\ref{im:w5} depicts a subgraph complementation system $\mathscr{C}$ for $W_5$ of cardinality 3, which is optimal by Theorem~\ref{thm:forbind2}. Since the class of graphs with subgraph complementation number at most $k$ is hereditary, we have $\cb(G) \geq 3$, and one might expect that $\cb(G)=3+1=4$, which is achieved by taking the union of the subgraph complementation systems for the components of $G$. However, since every vertex of $W_5$ appears an even number of times in $\mathscr{C}$, we can add the endpoints of the isolated edge in $G$ to each set in $\mathscr{C}$ to obtain a subgraph complementation system for $G$ of cardinality 3. That is, $\cb(G) = 3$.
\end{ex}

\begin{figure}
 \centering
 \scalebox{1.7}{
\begin{tikzpicture}
[node/.style={circle, draw=black!100, fill=black!100, inner sep=0pt, minimum size = 5pt}]

\node[node] (1) at (.5,.5) {};
\node[node] (2) at (0,0) {};
\node[node] (3) at (1,0) {};
\node[node] (4) at (1,1) {};
\node[node] (5) at (0,1) {};

\draw[black!20,line width=4pt] (2) -- (3) -- (4) -- (5) -- (2);
\draw[black!20,line width=4pt] (2) edge[bend right] (4);
\draw[black!20,line width=4pt] (3) edge[bend left] (5);

\draw[black!100,thick] (2) -- (3) -- (4) -- (5) -- (2);
\draw[black!100,thick] (2) edge[bend right] (4);
\draw[black!100,thick] (3) edge[bend left] (5);
\filldraw[color=black!100, fill=white!100,thick] (1) circle (.09);
\end{tikzpicture}
\hspace{2em}
\begin{tikzpicture}
[node/.style={circle, draw=black!100, fill=black!100, inner sep=0pt, minimum size = 5pt}]

\node[node] (1) at (.5,.5) {};
\node[node] (2) at (0,0) {};
\node[node] (3) at (1,0) {};
\node[node] (4) at (1,1) {};
\node[node] (5) at (0,1) {};

\draw[black!20,line width=4pt] (2) -- (1) -- (4);
\draw[black!20,line width=4pt] (2) edge[bend right] (4);

\draw[black!100,thick] (2) -- (3) -- (4) -- (5) -- (2);
\draw[black!100,thick] (2) -- (1) -- (4);

\filldraw[color=black!100, fill=white!100,thick] (3) circle (.09);
\filldraw[color=black!100, fill=white!100,thick] (5) circle (.09);

\end{tikzpicture}
\hspace{2em}
\begin{tikzpicture}
[node/.style={circle, draw=black!100, fill=black!100, inner sep=0pt, minimum size = 5pt}]

\node[node] (1) at (.5,.5) {};
\node[node] (2) at (0,0) {};
\node[node] (3) at (1,0) {};
\node[node] (4) at (1,1) {};
\node[node] (5) at (0,1) {};

\draw[black!20,line width=4pt] (3) -- (1) -- (5);
\draw[black!20,line width=4pt] (3) edge[bend left] (5);

\draw[black!100,thick] (2) -- (5);
\draw[black!100,thick] (2) -- (3);
\draw[black!100,thick] (5) -- (4);
\draw[black!100,thick] (3) -- (4);
\draw[black!100,thick] (1) -- (4);
\draw[black!100,thick] (1) -- (3);
\draw[black!100,thick] (1) -- (5);
\draw[black!100,thick] (1) -- (2);

\filldraw[color=black!100, fill=white!100,thick] (2) circle (.09);
\filldraw[color=black!100, fill=white!100,thick] (4) circle (.09);
\end{tikzpicture}
}
\caption{A subgraph complementation system for $W_5$.}
\label{im:w5}
\end{figure}
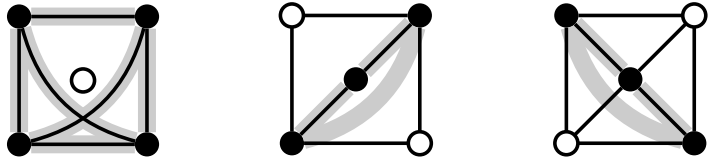

\begin{lemma}\label{lem:evensymdif}
Let $G$ be a graph, and let $v \in V(G)$. If $\mathscr{C}$ is a subgraph complementation system for $G$ in which every vertex in $V(G)\setminus \{v\}$ appears an even number of times, then the collection $\mathscr{C}_v$, which consists of the symmetric differences of $\{v\}$ with each set in $\mathscr{C}$, is also a subgraph complementation system for $G$.
\end{lemma}

\begin{proof}
Let $G$ be a graph, let $v \in V(G)$, and let $\mathscr{C}$ be a subgraph complementation system for $G$ in which every vertex in $V(G) \setminus \{v\}$ appears an even number of times in $\mathscr{C}$. Let $\mathscr{C}_v$ denote the collection of symmetric differences of $\{v\}$ with each set in $\mathscr{C}$, {\em i.e.}, $\mathscr{C}_v = \{C\triangle \{v\}: C\in \mathscr{C}\}$. For any $u \in V(G) \setminus \{v\}$, if $u$ and $v$ are contained in an odd number of sets together in $\mathscr{C}$, then $u$ is contained in an odd number of sets without $v$ in $\mathscr{C}$, so $u$ and $v$ appear together an odd number of times in $\mathscr{C}_v$. Similarly, if $u$ and $v$ appear together an even number of times in $\mathscr{C}$, then $u$ appears an even number of times without $v$ in $\mathscr{C}$, and thus an even number of times with $v$ in $\mathscr{C}_v$. Also, any two vertices which are distinct from $v$ appear together the same number of times in $\mathscr{C}_v$ as in $\mathscr{C}$. In other words, $\mathscr{C}_v$ is also a subgraph complementation system for $G$, as desired.
\end{proof}

In a particular case of Lemma~\ref{lem:evensymdif}, if every vertex of $G$ appears an even number of times in a subgraph complementation system $\mathscr{C}$, then for any $v \in V(G)$, the collection $\mathscr{C}_v$ is also a subgraph complementation system for $G$.

\begin{lemma}\label{lem:fried}{\rm \cite{friedland2012minimum}}
Let $A$ be an $n\times n$ symmetric matrix over $\mathbb{F}_2$ of rank $k$. Then either 
\begin{itemize}
 \item[i.] $A = XX^T$, where $X$ is an $n\times k$ matrix over $\mathbb{F}_2$ of rank $k$; or
 \item[ii.] $A = X (\oplus_1^l H_2) X^T$, where $X$ is as in i., $H_2 = \left(\begin{smallmatrix} 0&1 \\ 1&0 \end{smallmatrix}\right)$, and $k=2l$, so that the rank of $A$ is even.
\end{itemize}
\end{lemma}

Each of these cases can be interpreted combinatorially.
Let $G$ be the graph which $A$ fits. Case~{\it i.} of Lemma~\ref{lem:fried} will be helpful in establishing a close relationship between $c_2(G)$ and $\mr(G,\mathbb{F}_2)$ (see Corollary~\ref{cor:mrcb}). Case~{\it ii.}, while not related to subgraph complementation, can be interpreted in graph theoretic terms by a collection of $l$ complete tripartite graphs, or tricliques, on subsets of $V$ whose symmetric difference of edge sets is $E$. In Theorem~\ref{thm:mrcbt2}, we show that $\mr(G,\mathbb{F}_2)$ is either $c_2(G)$ or twice the minimum cardinality of such a collection of tricliques.  

Lemma~\ref{lem:fried} is a special case of Theorem~2.6 in~\cite{friedland1991quadratic}. It follows from the proof of this theorem that an $n\times n$ symmetric matrix $A$ decomposes as in case~{\em i.}\:if some diagonal entry is nonzero, and as in case~{\em ii.}\:if every diagonal entry is zero. There is a converse to this statement which will be of use to us.

\begin{prop}\label{prop:strongfried}
An $n \times n$ symmetric matrix $A = (a_{i,j})$ over $\mathbb{F}_2$ of rank $k$ decomposes as in case~{\it i.}\:of Lemma~\ref{lem:fried} if and only if $a_{i,i} = 1$ for some $i \in [n]$, and as in case~{\it ii.}\:if and only if $a_{i,i} = 0$ for all $i \in [n]$.
\end{prop}

\begin{proof}
Let $A$ be as described, and suppose that $A = X \left(\oplus_1^l \left(\begin{smallmatrix} 0&1 \\ 1&0 \end{smallmatrix}\right)\right) X^T$ for some $n\times 2l$ matrix $X = (x_{i,j})$. Then, for each $i \in [n]$, $a_{i,i} = 2x_{i,1}x_{i,2} + 2x_{i,3}x_{i,4} + \cdots + 2x_{i,2l-1}x_{i,2l} \equiv 0 \pmod{2}$. That is, every diagonal entry of $A$ is zero. On the other hand, suppose that $A = XX^T$. A diagonal entry $a_{i,i}$ is $0$ if and only if the $i$th row of $X$ has an even number of $1$'s. Thus, if every $a_{i,i} = 0$, the columns of $X$ are linearly dependent over $\mathbb{F}_2$, so that $\rank(X)$, which is at least $k$ since $\rank(X) \geq \rank(XX^T)$, is strictly less than the number of columns of $X$. Thus, $A$ decomposes as in case~{\em i.} of Lemma~\ref{lem:fried} if and only if some diagonal entry of $A$ is nonzero.
\end{proof}

Let $A$ be a symmetric $n\times n$ matrix, and let $G$ be the graph which $A$ fits. If $A = XX^T$, with $X$ as in case~{\em i.} of Lemma~\ref{lem:fried}, then the rows of $X$ constitute a faithful orthogonal representation of $G$ over $\mathbb{F}_2$. As we saw in the introduction, we can interpret the rows of $X$ as incidence vectors for a subgraph complementation system $\{C_1, C_2, \ldots, C_k\}$ for $G$; for $i = 1, 2, \ldots, k$, include a vertex $v$ in $C_i$ if and only if the $i$th entry of the row associated to $v$ in $X$ is 1.

\begin{thm}\label{thm:ourfried}
Let $A$ be an $n\times n$ symmetric matrix over $\mathbb{F}_2$ of rank $k$. Then either
\begin{enumerate}[i.]
    \item $A = XX^T$, where $X$ is an $n \times k$ matrix over $\mathbb{F}_2$ of rank $k$; or
    \item $A = XX^T$, where $X$ is an $n \times (k+1)$ matrix over $\mathbb{F}_2$ of rank $k$, and $k$ is even.
\end{enumerate}
\end{thm}

\begin{proof}
Let $A$ be as described, and let $G$ be the graph which $A$ fits. Suppose that we are in case~{\em ii.} of Lemma~\ref{lem:fried}. Then $k$ is even, and, by Proposition~\ref{prop:strongfried}, every diagonal entry of $A$ is $0$. Let $B$ be the matrix obtained by changing a single diagonal entry of $A$ from $0$ to $1$. Then $B$ also fits $G$. It is not hard to see that the ranks of $A$ and $B$ differ by no more than 1. Furthermore, since $B$ has a nonzero diagonal entry, $B = YY^T$ for some $n \times \rank(B)$ matrix $Y$ by Proposition~\ref{prop:strongfried}. Let $\mathscr{C}$ be the subgraph complementation system associated to $Y$. 

Suppose that $\rank(B) = k+1$. Let $v$ be the vertex corresponding to the row in which $B$ has a nonzero diagonal entry. By Lemma~\ref{lem:evensymdif}, the collection $\mathscr{C}_v$ consisting of the symmetric differences of $\{v\}$ with every set in $\mathscr{C}$ is a subgraph complementation system for $G$. Since $|\mathscr{C}|$ is odd, every vertex of $G$ appears an even number of times in $\mathscr{C}_v$. Thus, if $Z$ is the matrix associated to $\mathscr{C}_v$, then $Z$ is an $n\times (k+1)$ matrix such that $ZZ^T = A$. Each row of $Z$ contains an even number of $1$'s, so the columns of $Z$ are linearly dependent, and $\rank(Z) \leq k$. Equality follows, as $\rank(Z) \geq \rank(A) = k$. Otherwise, $\rank(B) \leq k$. In this case, the collection $\mathscr{C} \cup \{\{v\}\}$ is a subgraph complementation system for $G$ of cardinality at most $k+1$ in which every vertex appears an even number of times. If $M$ is the matrix associated to $\mathscr{C} \cup \{\{v\}\}$, then $MM^T = A$. Since $\rank(M) \geq \rank(A)$, and since the columns of $M$ are linearly dependent, $M$ must have $k+1$ columns and rank $k$. This completes the proof.
\end{proof}

Suppose that $A$ is a matrix which fits a graph $G$ of rank $k = \mr(G,\mathbb{F}_2)$. Then the subgraph complementation system associated to the matrix $X$ obtained in Theorem~\ref{thm:ourfried} has cardinality either $k$ or $k+1$, depending on whether we are in case~{\em i.} or case~{\em ii.}. This implies a close relationship between $\mr(G,\mathbb{F}_2)$ and $\cb(G)$. 
\par

\begin{cor}\label{cor:mrcb}
Let $G$ be a graph. Then either 
\begin{itemize}
 \item[i.] $\cb(G) = \mr(G,\mathbb{F}_2)$, or 
 \item[ii.] $\cb(G) = \mr(G,\mathbb{F}_2)+1$, in which case $\mr(G,\mathbb{F}_2)$ is even.
\end{itemize}
\end{cor}

There is a simpler proof of Corollary~\ref{cor:mrcb} which compares the additivity of the minimum rank function over $\mathbb{F}_2$ to the subadditivity of the subgraph complementation number. Let $k = \mr(G,\mathbb{F}_2)$, and suppose that $\cb(G)\neq k$, so that $\mr(G,\mathbb{F}_2)<\cb(G)$. Then $k$ is even by Lemma~\ref{lem:fried}, otherwise there would exist an $n \times k$ matrix $X$ over $\mathbb{F}_2$ such that $XX^T$ has rank $k$ and fits $G$, and the associated subgraph complementation system for $G$ would have cardinality $k$. Consider $G+K_2$. By the additivity of the minimum rank of a graph, $\mr(G+K_2,\mathbb{F}_2)=\mr(G,\mathbb{F}_2)+1$, which is odd. Thus,
\[\cb(G)\leq \cb(G+K_2)=\mr(G+K_2,\mathbb{F}_2)=\mr(G,\mathbb{F}_2)+1,\]
as desired.

We proceed to characterize the graphs for which $\cb(G)=\mr(G,\mathbb{F}_2)+1$. A characterization of the adjacency matrices of such graphs follows directly from Theorem~\ref{thm:ourfried}.
\par

\begin{thm}\label{thm:mrzerodiag}
Let $G$ be a nonempty graph of minimum rank $k$ over $\mathbb{F}_2$. Then $\cb(G) \neq k$ if and only if the adjacency matrix of $G$ has rank $k$, and every other matrix which fits $G$ over $\mathbb{F}_2$ has rank strictly larger than $k$.
\end{thm}

\begin{proof}
Let $G$ be a nonempty graph of minimum rank $k$ over $\mathbb{F}_2$, and let $A$ be a matrix which fits $G$ over $\mathbb{F}_2$ of rank $k$. By Theorem~\ref{thm:ourfried}, $A = XX^T$ for some $n\times k$ or $n\times (k+1)$ matrix $X$. By Proposition~\ref{prop:strongfried}, there exists such an $n\times k$ matrix $X$ if and only if some diagonal entry of $A$ is nonzero. Every such matrix $X$ corresponds to a subgraph complementation system for $G$ whose cardinality is the number of columns of $X$, from which we obtain the desired result.
\end{proof}

We will now characterize the subgraph complementation systems of graphs for which $\cb(G) = \mr(G,\mathbb{F}_2) + 1$. We start with the following lemma.

\begin{lemma}\label{lem:evencboddvertex}
Let $G$ be a graph with $\cb (G)$ even, and let $\mathscr{C}$ be a minimum subgraph complementation system for $G$. Then there exists a vertex $v \in V$ such that $v$ appears in $\mathscr{C}$ an odd number of times.
\end{lemma}

\begin{proof}
Let $G$ and $\mathscr{C}$ be as described. Suppose, for the sake of contradiction, that every vertex of $G$ appears an even number of times in $\mathscr{C}$. Let $C=\{ u_1,\dots,u_s \}$ be a set in $\mathscr{C}$. Then $\mathscr{C}_{u_1}$ is a minimum subgraph complementation system for $G$, by Lemma~\ref{lem:evensymdif}. Furthermore, $\mathscr{C}_{u_1}$ maintains the property that every vertex appears an even number of times. We can continue this process to find that $\mathscr{C}_{u_1,u_2}=(\mathscr{C}_{u_1})_{u_2}$ also maintains that property, and so on. Then $\mathscr{C}_{u_1,\dots,u_s}$ is a minimum subgraph complementation system for $G$, but it contains the empty set $C$. This implies that $\mathscr{C}_{u_1,\dots,u_s}\setminus \{C\}$ is also a subgraph complementation system for $G$, which contradicts the minimality of $\mathscr{C}$.
\end{proof}

If $\mathscr{C}$ is a subgraph complementation system of odd cardinatlity in which every vertex appears an even number of times, then the vertex $v$ appears an odd number of times in the subgraph complementation system $\mathscr{C}_v$ from Lemma~\ref{lem:evensymdif}. Together with Lemma~\ref{lem:evencboddvertex}, we see that, for any graph $G$, there exists a minimum subgraph complementation system in which some vertex appears an odd number of times. We will show that this is the case for every minimum subgraph complementation system if and only if $\cb(G)=\mr(G,\mathbb{F}_2)$.
\par

\begin{thm}\label{thm:mrneqcbevenvxs}
Let $G$ be a nonempty graph. Then $\cb(G) \neq \mr(G,\mathbb{F}_2)$ if and only if there exists a minimum subgraph complementation system $\mathscr{C}$ for $G$ in which every vertex of $G$ appears an even number of times.
\end{thm}

\begin{proof}
Let $G$ be a graph with at least one edge, and let $k = \mr(G,\mathbb{F}_2) > 0$. We begin by proving sufficiency. Supppose that there exists a minimum subgraph complementation system $\mathscr{C}$ of $G$ in which every vertex appears an even number of times. Let $X$ be the matrix associated to $\mathscr{C}$. Each row of $X$ contains an even number of $1$'s, so the columns of $X$ are linearly dependent. Thus, \[k \leq \rank(XX^T) \leq \rank(X) < c_2(G).\]
\par

Concerning the necessary condition, suppose that $\cb(G) \neq k$. Then $\cb(G) = k+1$ by Corollary~\ref{cor:mrcb}. By Theorem~\ref{thm:mrzerodiag}, the adjacency matrix of $G$, $A = A(G)$, is the unique matrix which fits $G$ of minimum rank over $\mathbb{F}_2$. By Proposition~\ref{prop:strongfried} and Theorem~\ref{thm:ourfried}, $A = XX^T$ for some $n \times (k+1)$ matrix $X$ of rank $k$ over $\mathbb{F}_2$. Then every row of $X$ has an even number of $1$'s. Taking the rows of $X$ as incidence vectors, we obtain a subgraph complementation system for $G$ in which every vertex appears an even number of times, as desired. 
\end{proof}

\begin{thm}\label{thm:mrneqcbcomponents}
Let $G$ be a nonempty graph with components $G_1, \ldots, G_t$. Then $\cb(G) \neq \mr(G,\mathbb{F}_2)$ if and only if $\cb(G_i) \neq \mr(G_i,\mathbb{F}_2)$ for all $i \in [t]$.
\end{thm}

\begin{proof}
Let $G = G_1 + \cdots + G_t$. If $\mr(G,\mathbb{F}_2) \neq \cb(G)$, by Theorem~\ref{thm:mrzerodiag}, the adjacency matrix $A = A(G)$ is the unique matrix which fits $G$ of minimum rank over $\mathbb{F}_2$. Suppose, for the sake of contradiction, that there exists a component $G_k$ of $G$ for which $\mr(G_k,\mathbb{F}_2) = \cb(G_k)$. Notice that every matrix which fits $G$ is a block-diagonal matrix; let $A = \oplus_1^t A_i$. Furthermore, the rank of a block-diagonal matrix is minimized by minimizing the ranks of its blocks, so that $\rank(A_i) = \mr(G_i,\mathbb{F}_2)$ for each $i \in [t]$. By Theorem~\ref{thm:mrneqcbevenvxs}, there exists a minimum subgraph complementation system $\mathscr{C}$ for $G_k$ in which some vertex appears an odd number of times. Let $M=M(\mathscr{C})$ be the matrix associated to $\mathscr{C}$. Then $MM^T$ fits $G_k$, is of rank $\mr(G_k,\mathbb{F}_2)$, and has some nonzero diagonal entry. We may thus replace $A_k$ by $MM^T$ to obtain a matrix fitting $G$ of minimum rank over $\mathbb{F}_2$ with a nonzero diagonal entry, a contradiction.
\par

On the other hand, if $\mr(G_i,\mathbb{F}_2) \neq \cb(G_i)$ for every $i \in [t]$, then, for each $i$, the adjacency matrix $A_i = A(G_i)$ is the unique matrix of minimum rank over $\mathbb{F}_2$ which fits $G_i$. Thus, there is a unique matrix fitting $G$ over $\mathbb{F}_2$ of minimum rank, and it consists of the blocks $A_i$ for $i \in [t]$. By Theorem~\ref{thm:mrzerodiag}, we have $\cb(G) \neq \mr(G,\mathbb{F}_2)$, as desired.
\end{proof}

We summarize our characterization of the graphs for which $\cb(G) \neq \mr(G,\mathbb{F}_2)$ in the following theorem.
\par

\begin{thm}\label{thm:mrneqcbsummary}
Let $G$ be a nonempty graph. The following are equivalent.
\begin{enumerate}[i.]
    \item $\cb(G) \neq \mr(G,\mathbb{F}_2)$;
    \item $\cb(G) = \mr(G,\mathbb{F}_2) + 1$;
    \item there is a unique matrix $A$ of minimum rank over $\mathbb{F}_2$ which fits $G$, and every diagonal entry of $A$ is 0;
    \item there is a minimum subgraph complementation system for $G$ in which every vertex appears an even number of times;
    \item for every component $G'$ of $G$, $\cb(G') = \mr(G', \mathbb{F}_2) + 1$.
\end{enumerate}
\end{thm}

\subsection{Tripartite subgraph complementation}

In addition to subgraph complementation, the authors of~\cite{kaminski2009recent} also defined {\em bipartite subgraph complementation}, the operation of complementing the edges between disjoint subsets $A$ and $B$ of vertices of a graph $G$. Equivalently, it is the operation of taking the symmetric difference of the edge set of $G$ and that of the complete bipartite graph with partite sets $A$ and $B$. In this section, we provide an alternative characterization of the graphs for which $\cb(G) = \mr(G,\mathbb{F}_2) + 1$ by extending the notion of bipartite subgraph complementation. We define {\em tripartite subgraph complementation} to be the operation of complementing the edges between three (possibly empty) disjoint subsets of vertices of $G$. The addition or removal of an edge of $G$ is a special case of tripartite subgraph complementation, as it is for both subgraph complementation and bipartite subgraph complementation. This leads us to define a parameter $t_2(G)$ similarly to $\cb(G)$; $t_2(G)$ is the minimum number of tripartite subgraph complementations needed to obtain $G$ from $\overline{K}_n$. Equivalently, $t_2(G)$ is the minimum number of complete tripartite graphs on subsets of $V$ such that each pair of vertices are adjacent in an odd number of complete tripartite graphs if and only if they are adjacent in $G$, or, whose symmetric difference of edge sets is $E$.

In Theorem~\ref{thm:mrcbt2}, we show that for all graphs with $\cb(G) = \mr(G,\mathbb{F}_2) + 1$, we have $\mr(G,\mathbb{F}_2) = 2t_2(G)$. In general, we will see that $\mr(G,\mathbb{F}_2) \leq 2t_2(G)$. In order to understand this relationship, we need a different interpretation of the second case of Lemma~\ref{lem:fried}.
For the remainder of this section, we define \[H_2 = \left(\begin{matrix} 0 & 1 \\ 1 & 0 \end{matrix}\right).\]

\begin{ex}\label{ex:W5}
In Example~\ref{ex:cbneqmr}, we saw that $\cb(K_{3,3}) = \mr(K_{3,3},\mathbb{F}_2) + 1$. The same is true of the wheel graph $W_5$, depicted in Figure~\ref{im:w5}, which is the smallest graph with this property. Let $v_1$ denote the center vertex of the wheel, and $v_2, \ldots, v_5$ denote the vertices around the rim, labeled cyclically. By Theorem~\ref{thm:mrzerodiag}, the adjacency matrix $A = A(W_5)$ is the only matrix of rank $\mr(W_5,\mathbb{F}_2) = 2$ which fits $W_5$. While there does not exist a $5\times 2$ matrix $X$ such that $A = XX^T$, the matrix
\[X = \left( \begin{matrix} 
1 & 1 \\ 1 & 0 \\ 0 & 1 \\ 1 & 0 \\ 0 & 1
\end{matrix}\right)\]
is such that $A = X H_2 X^T$, in accordance with Lemma~\ref{lem:fried}. Notice that $W_5$ is actually a complete tripartite graph, and that, if we label the partite sets by the vectors $(1,1)$, $(1,0)$, and $(0,1)$, $X$ may be seen as an incidence matrix for the partite sets of $W_5$. In fact, to any complete tripartite graph $G$ on $n$ vertices we can associate such an incidence matrix $M$ with $M H_2 M^T = A(G)$ (including $K_{3,3}$, which is a triclique with an empty partite set). We will extend this notion to obtain Theorem~\ref{thm:mrcbt2}.
\end{ex}

Let $G$ be a graph with vertex set $V = \{v_1, \ldots, v_n\}$ and with adjacency matrix $A = A(G)$ of rank $k$ over $\mathbb{F}_2$. By Proposition~\ref{prop:strongfried}, there exists an $n \times k$ matrix $X$ of rank $k$ such that $X (\oplus_1^{l} H_2) X^T = A$, where $k = 2l$. We group the columns of $X$ into pairs and denote the entries by
\[X = \left( \begin{matrix}
x_{11} & y_{11} & \cdots & x_{1l} & y_{1l} \\
x_{21} & y_{21} & \cdots & x_{2l} & y_{2l} \\
\vdots & \vdots & \ddots & \vdots & \vdots \\
x_{n1} & y_{n2} & \cdots & x_{nl} & y_{nl} 
\end{matrix} \right).\]
The $(i,j)$th entry of $A = X (\oplus_1^l H_2) X^T$ is
\[a_{ij} = \sum_{m=1}^l (x_{im}y_{jm} + x_{jm}y_{im}) \pmod{2},\]
of which the $m$th summand $x_{im}y_{jm} + x_{jm}y_{im}$ is 1 if and only if $x_{im}y_{jm} \neq x_{jm}y_{im}$.

Consider the collection $\mathscr{T} = \{T_1, \ldots, T_l\}$ of complete tripartite graphs on subsets of $V$, with partite sets $(X_m, Y_m, Z_m)$ for each $m \in [l]$, such that 
\[v_i \in
\begin{cases}
X_m : & \mbox{if $x_{im} = 1$ and $y_{im} = 0$;} \\
Y_m : & \mbox{if $x_{im} = 0$ and $y_{im} = 1$;} \\
Z_m : & \mbox{if $x_{im} = y_{im} = 1$;}
\end{cases}
\]
and $v_i \not\in T_m$ if $x_{im} = y_{im} = 0$. Then $v_iv_j \in E(T_m)$ if and only if $x_{im}y_{jm} \neq x_{jm}y_{im}$. Since $A = X (\oplus_1^l H_2) X^T$, we see that a pair of vertices are adjacent in $G$ if and only if they are adjacent in an odd number of tricliques in $\mathscr{T}$. Conversely, given a collection of $l$ tricliques on subsets of $V$ in which each pair $v_i, v_j \in V$ is adjacent in an odd number of tricliques if and only if $v_iv_j \in E(G)$, we can construct an $n\times 2l$ matrix $X$ in the same fashion so that $A = X (\oplus_1^l H_2) X^T$.

\begin{thm}\label{thm:mrcbt2}
For any graph $G$, we have
\[\mr(G,\mathbb{F}_2) = \min\{\cb(G), 2t_2(G)\}.\]
\end{thm}
\begin{proof}
Let $G$ be a graph with adjacency matrix $A = A(G)$, and let $\mr(G,\mathbb{F}_2) = k$. Suppose that $\cb(G) \neq k$. By Theorem~\ref{thm:mrzerodiag}, $A$ is the only matrix of rank $k$ over $\mathbb{F}_2$ which fits $G$. By Lemma~\ref{lem:fried}, there exists an $n \times k$ matrix $X$ of rank $k$ over $\mathbb{F}_2$ such that $A = X(\oplus_1^l H_2) X^T$, where $k = 2l$ and $H_2 = \left(\begin{smallmatrix} 0&1 \\ 1&0 \end{smallmatrix}\right)$. Let $\mathscr{T}$ be the collection of $l$ tricliques induced by $X$. Then each pair of vertices are adjacent in $G$ if and only if they are adjacent in an odd number of tricliques in $\mathscr{T}$. Thus, $t_2(G) \leq |\mathscr{T}| =  k/2$. On the other hand, the rank of a matrix $M$ such that $A = M (\oplus_1^l H_2) M^T$ is at least the rank of $A$, since any vector in the span of $A$ is also in the span of $M$. Thus, $2t_2(G) \geq k$, which completes the proof.
\end{proof}

\subsection{Forests}\label{sec:forests}

The fact that the subgraph complementation number of a graph is at least its minimum rank over $\mathbb{F}_2$ may help us to determine the subgraph complementation number of a graph whose minimum rank over $\mathbb{F}_2$ is known. For example, the minimum rank problem over $\mathbb{R}$ is solved for trees, in the sense that the problem has been reduced from finding the minimum rank of a matrix from an infinite class to finding the optimal value of some graph parameter on a finite number of vertices (see \cite{HOGBEN20101961} for a survey). The problem is also solved for forests by the additivity of $\mr(G,\mathbb{R})$. Furthermore, the authors of~\cite{chenette2007minimum} have shown that the minimum rank of a tree is independent of the field, which can similarly be generalized to forests. Throughout this section, when $G$ is a forest, we will thus refer to $\mr(G)$ without confusion. In Theorem~\ref{cor:mrcbtree}, we prove equality of the minimum rank of a forest and its subgraph complementation number. We begin by showing this equality for linear forests.
\par

\begin{prop}\label{prop:linearforest}
If $L$ is a linear forest with $k$ components, then $\cb(L) = n-k$.
\end{prop}

\begin{proof}
Let $L$ be a linear forest with components $P^{(1)}, \ldots, P^{(k)}$. It is clear that $\mathscr{C} = \{\{u,v\}\mid uv\in E(L)\}$ of cardinality $\|L\| = n-k$ is a subgraph complementation system for $L$, so $\cb(L)\leq n-k$. It follows from Fiedler's Tridiagonal Matrix Theorem~\cite{fiedler1969characterization} that, for any tree $T$, $\mr(T) = n-1$ if and only if $T\sim P_n$. By the additivity of the minimum rank function, $\mr(L) = \sum_1^k \mr(P^{(i)}) = n-k$. We have seen in Section~\ref{sec:mrgeneralgraphs}, equation~\eqref{eq:mrlower}, that $\mr(L)\leq \cb(L)$, from which the result follows.
\end{proof}
 
In general, there is no straightforward relationship between $\mr (G,\mathbb{F}_2)$ and $\mr (G,\mathbb{R})$. The smallest example is the full house graph, depicted in Figure~\ref{im:forbind2}, which has minimum rank 3 over $\mathbb{F}_2$, but minimum rank 2 over any other field~\cite{barrett2009minimum}. On the other hand, the complete tripartite graph $K_{3,3,3}$ is one of the minimal forbidden induced subgraphs for the class $\{G \mid \mr(G,\mathbb{R})\leq 2\}$, but its adjacency matrix with zeros on the diagonal has rank 2 over $\mathbb{F}_2$. 
\par

We can use these examples and the additive property of minimum rank to construct examples of graphs $G$ where $\mr (G,\mathbb{F}_2)$ and $\mr (G,\mathbb{R})$ are arbitrarily far apart in either direction. We now know that $\mr (G,\mathbb{F}_2)$ and $\cb (G)$ are not necessarily equal, but always close. Their relationship with $\mr (G,\mathbb{R})$ is difficult to pin down for general graphs. However, when $F$ is a forest, $\mr (F,\mathbb{F}_2)$, $\mr (F,\mathbb{R})$ and $\cb (F)$ coincide. We conclude this section by connecting the subgraph complementation number of a forest to its {\em path cover number} $\pc(G)$, or the minimum cardinality of a collection of vertex-disjoint induced paths which cover all of the vertices of $G$. Such a collection is called a {\em path cover} of $G$. It is known that, when $T$ is a tree, $\mr(T,\mathbb{F}) = |T|-\pc(T)$ for any field $\mathbb{F}$~\cite{chenette2007minimum}. The case $\mathbb{F} = \mathbb{R}$ was proven in \cite{johnson1999maximum}.
\par

\begin{thm}\label{cor:mrcbtree}
For any forest $F$ and field $\mathbb{F}$, we have \[\cb(F)=\mr(F,\mathbb{F})=|F|-\pc (F).\]
\end{thm}

\begin{proof} We prove the result for trees, and obtain the result for forests by the addivity of the minimum rank of a graph. Let $T$ be a tree. By equation~\eqref{eq:mrlower}, $\mr(T)\leq \cb(T)$. We will show that $\cb (T) \leq |T|-\pc(T) = \mr(T)$ by finding a minimum subgraph complementation system from a minimum path cover of $T$. Hogben and Johnson presented an algorithm for finding a minimum path cover for trees~\cite{hogbenpath} (see also~\cite{fallat2007minimum}). If $\mathcal{P}$ is a path cover of $T$ produced by this algorithm, then every path in $\mathcal{P}$ contains at most one {\em high-degree vertex}, a vertex of degree 3 or more in $T$, and these vertices are never endpoints of the paths in which they lie. If $\mathcal{P}$ also covers all of the edges of $T$, then $T$ is a linear forest and Proposition~\ref{prop:linearforest} completes the proof. Otherwise, any edges which are not in $\mathcal{P}$ are adjacent to high-degree vertices, which are internal in their respective paths. Denote these high-degree vertices by $v_1,v_2,\ldots,v_k$, and define $U = \{v\in V(T)\mid d_T(v)\leq 2\}$. Let $\mathscr{C}$ be the collection consisting of the $|E(\mathcal{P})|-2k$ sets of the form $\{u,v\}$ where $u,v\in U$ and $uv\in E(\mathcal{P})$, along with the sets $N_T(v_i)$ and $N_T[v_i]$ for $1\leq i\leq k$. Then $\mathscr{C}$ is a subgraph complementation system for $T$ of cardinality $|E(\mathcal{P})| = |T|-\pc(T)$. Therefore, $\cb(T)\leq |T|-\pc(T) = \mr(T)$, which completes the proof.
\end{proof}

\section{Forbidden induced subgraphs}\label{sec:forb}

The class of graphs with subgraph complementation number at most $k$ is hereditary for any nonnegative integer $k$. We have seen that $\mr(G, \mathbb{F}_2)\leq \cb(G)$ in general, implying that 
\begin{equation}\label{eq:cbleqksubsetmrleqk}
    \{ G \mid \cb (G)\leq k \} \subseteq \{ G \mid \mr (G,\mathbb{F}_2)\leq k \}.
\end{equation}
It is known that the class of graphs $\{ G \mid \mr (G, \mathbb{F})\leq k \}$ is hereditary and finitely defined when $\mathbb{F}$ is finite~\cite{ding2006minimal}. For odd $k$, it follows from Corollary~\ref{cor:mrcb} that if $\mr(G,\mathbb{F}_2)=k$, then $\cb(G)=k$, and if $\mr(G,\mathbb{F}_2)<k$, then $\cb(G)\leq k$. Therefore, when $k$ is odd, we also have $\{ G \mid \cb (G)\leq k \} \supseteq \{ G \mid \mr (G,\mathbb{F}_2)\leq k \}$.
\par

\begin{prop}\label{prop:oddfindef}
For any odd $k$, 
\[\{G\mid\cb(G)\leq k\}=\{G\mid\mr(G,\mathbb{F}_2)\leq k\}.\]
\end{prop}

In particular, the classes $\{G\mid\cb(G)\leq k\}$ and $\{G\mid\mr(G,\mathbb{F}_2)\leq k\}$ for odd $k$ are defined by the same finite set of minimal forbidden induced subgraphs. The two minimal forbidden induced subgraphs for $k=1$ are evident, as a graph with $c_2(G) \leq 1$ consists of a single clique and/or isolated vertices. That is, the class of graphs $\{ G \mid \cb(G)\leq 1 \}$ is the class of $\{ P_3, 2K_2\}$-free graphs. We obtain as a corollary to Proposition~\ref{prop:oddfindef} that the set of minimal forbidden induced subgraphs for the property $\cb(G)\leq 3$ is the same set given in the following theorem and listed explicitly in~\cite{barrett2009minimum}.
\par

\begin{thm}\label{thm:forbind3}{\rm \cite{barrett2009minimum}}
The class of graphs
\[\{G\mid\mr(G,\mathbb{F}_2)\leq 3\}\]
is defined by forbidding a set of 62 minimal induced subgraphs, each of which has 8 or fewer vertices.
\end{thm}

On the other hand, when $k$ is even, it does not follow from Proposition~\ref{prop:oddfindef} that $\{ G \mid \cb (G)\leq k \}$ is finitely defined. We prove this in the following theorem.
\par

\begin{thm}\label{thm:cbfindef}
For any natural number $k$, the class of graphs 
$$\{ G \mid \cb(G)\leq k \}$$
is defined by forbidding a finite set of induced subgraphs.
\end{thm}

\begin{proof} Let $F$ be a minimal forbidden induced subgraph for the property $\cb(G)\leq k $. First, we claim that $\cb(F)\leq k+2 $. Suppose, for the sake of contradiction, that $\cb(F)\geq k+3 $. Then, for any $v \in V(F)$ and subgraph complementation system $\mathscr{C}'$ for $F-v$, we have that $\mathscr{C} = \mathscr{C}'\cup \{N(v),N[v]\}$ is a subgraph complementation system for $F$, which implies that $\cb(F-v)\geq k+1$. This contradicts the minimality of $F$. 
\par

Now, there exists a subgraph complementation system $\mathscr{C}$ for $F$ of cardinality $k+2$. We can associate to $F$ a vector of length $s=2^{k+2}$, where each entry corresponds to an element of the powerset $2^\mathscr{C}$, such that each entry of the vector is a non-negative integer that counts the number of vertices of $F$ that are in a given subcollection of $\mathscr{C}$. This vector defines the graph $F$ up to isomorphism. It is easy to verify that, if two graphs $F_a$ and $F_b$ have vectors $(a_1,\dots,a_s)$ and $(b_1,\dots,b_s)$ such that $a_i \leq b_i$ for $1 \leq i \leq s$, then $F_a$ is an induced subgraph of $F_b$. We now see that the poset of forbidden induced subgraphs for the property $\cb(G)\leq k $ ordered by the induced subgraph relation can be embedded in the poset $\mathbb{N}^s$, which is the direct product of the poset $\mathbb{N}$ ordered by $\leq$. It is known that a direct product of finitely many posets that are well-founded and that have no infinite anti-chains is itself well-founded and has no infinite anti-chains~\cite{kruskal1972theory}. Furthermore, any restriction of such a poset has the same properties. This completes the proof to show that the poset of forbidden induced subgraphs for the property $\cb(G)\leq k $, ordered by the induced subgraph relation, is well-founded with a finite number of minimal elements.
\end{proof}

Theorem~\ref{thm:cbfindef} only guarantees that the set of minimal forbidden induced subgraphs for the property $\cb(G)\leq k $ is finite; it does not provide an explicit upper bound. Based on the results concerning linear forests, we present the following conjecture.
\par

\begin{conj}
A minimal forbidden induced subgraph for the property $\cb(G) \leq k$ has at most $2k+2$ vertices.
\end{conj}

By analyzing the structure of graphs with $c_2(G) \leq 2$, we can find the set of minimal forbidden induced subgraphs for this property.  This is the set given in Theorem~\ref{thm:forbind2} and depicted in Figure~\ref{im:forbind2} (A).
\par

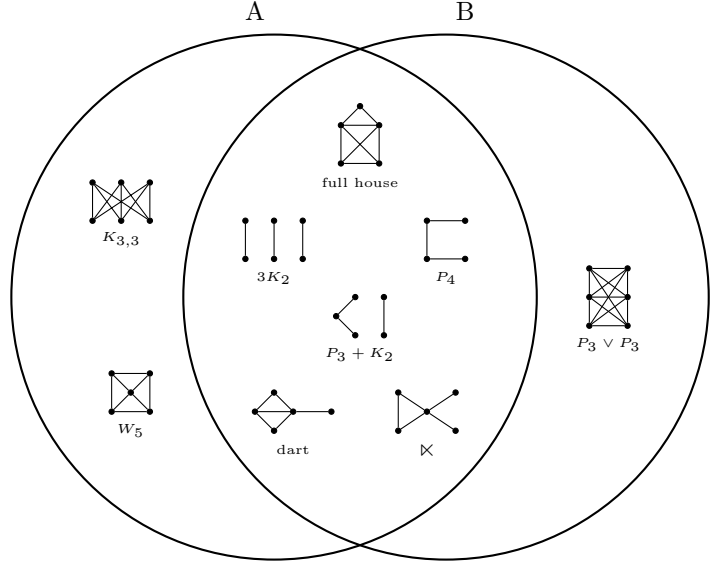
\begin{figure}
 \centering
\begin{tikzpicture}[x=.1in,y=.1in, every edge/.style = {draw, thick}]

\node[circle, draw, thick, inner sep=0pt, minimum size=2.75in] (cb) at (-4.5,0) {};
\node[circle, draw, thick, inner sep=0pt, minimum size=2.75in] (mr) at (4.5,0) {};
\node (A) at (-5.5,15) {A};
\node (B) at (5.5,15) {B};

\node[circle, draw, fill=black!100, inner sep=0pt, minimum size=2pt] (P40) at (5.5,2) {};
\node[circle, draw, fill=black!100, inner sep=0pt, minimum size=2pt] (P41) at (3.5,2) {};
\node[circle, draw, fill=black!100, inner sep=0pt, minimum size=2pt] (P42) at (3.5,4) {};
\node[circle, draw, fill=black!100, inner sep=0pt, minimum size=2pt] (P43) at (5.5,4) {};
\draw (P40) -- (P41) -- (P42) -- (P43);
\node (P4) at (4.5,1) {\tiny $P_4$};

\node[circle, draw, fill=black!100, inner sep=0pt, minimum size=2pt] (P3K20) at (-1.25,-1) {};
\node[circle, draw, fill=black!100, inner sep=0pt, minimum size=2pt] (P3K21) at (-.25,-2) {};
\node[circle, draw, fill=black!100, inner sep=0pt, minimum size=2pt] (P3K22) at (-.25,0) {};
\node[circle, draw, fill=black!100, inner sep=0pt, minimum size=2pt] (P3K23) at (1.25,-2) {};
\node[circle, draw, fill=black!100, inner sep=0pt, minimum size=2pt] (P3K24) at (1.25,0) {};
\draw (P3K21) -- (P3K20) -- (P3K22);
\draw (P3K23) -- (P3K24);
\node (P3K2) at (0,-3) {\tiny $P_3+K_2$};

\node[circle, draw, fill=black!100, inner sep=0pt, minimum size=2pt] (3K20) at (-6,2) {};
\node[circle, draw, fill=black!100, inner sep=0pt, minimum size=2pt] (3K21) at (-4.5,2) {};
\node[circle, draw, fill=black!100, inner sep=0pt, minimum size=2pt] (3K22) at (-3,2) {};
\node[circle, draw, fill=black!100, inner sep=0pt, minimum size=2pt] (3K23) at (-6,4) {};
\node[circle, draw, fill=black!100, inner sep=0pt, minimum size=2pt] (3K24) at (-4.5,4) {};
\node[circle, draw, fill=black!100, inner sep=0pt, minimum size=2pt] (3K25) at (-3,4) {};
\draw (3K20) -- (3K23);
\draw (3K21) -- (3K24);
\draw (3K22) -- (3K25);
\node (3K2) at (-4.5,1) {\tiny $3K_2$};

\node[circle, draw, fill=black!100, inner sep=0pt, minimum size=2pt] (FullHouse0) at (0,10) {};
\node[circle, draw, fill=black!100, inner sep=0pt, minimum size=2pt] (FullHouse1) at (-1,9) {};
\node[circle, draw, fill=black!100, inner sep=0pt, minimum size=2pt] (FullHouse2) at (-1,7) {};
\node[circle, draw, fill=black!100, inner sep=0pt, minimum size=2pt] (FullHouse3) at (1,7) {};
\node[circle, draw, fill=black!100, inner sep=0pt, minimum size=2pt] (FullHouse4) at (1,9) {};
\draw (FullHouse0) -- (FullHouse1) -- (FullHouse2) -- (FullHouse3) -- (FullHouse4) -- (FullHouse0);
\draw (FullHouse2) -- (FullHouse4) -- (FullHouse1) -- (FullHouse3);
\node (FullHouse) at (0,6) {\tiny full house};

\node[circle, draw, fill=black!100, inner sep=0pt, minimum size=2pt] (Dart0) at (-5.5,-6) {};
\node[circle, draw, fill=black!100, inner sep=0pt, minimum size=2pt] (Dart1) at (-3.5,-6) {};
\node[circle, draw, fill=black!100, inner sep=0pt, minimum size=2pt] (Dart2) at (-1.5,-6) {};
\node[circle, draw, fill=black!100, inner sep=0pt, minimum size=2pt] (Dart3) at (-4.5,-5) {};
\node[circle, draw, fill=black!100, inner sep=0pt, minimum size=2pt] (Dart4) at (-4.5,-7) {};
\draw (Dart0) -- (Dart1) -- (Dart2);
\draw (Dart0) -- (Dart3) -- (Dart1) -- (Dart4) -- (Dart0);
\node (Dart) at (-3.5,-8) {\tiny dart};

\node[circle, draw, fill=black!100, inner sep=0pt, minimum size=2pt] (ltimes0) at (2,-5) {};
\node[circle, draw, fill=black!100, inner sep=0pt, minimum size=2pt] (ltimes1) at (2,-7) {};
\node[circle, draw, fill=black!100, inner sep=0pt, minimum size=2pt] (ltimes2) at (3.5,-6) {};
\node[circle, draw, fill=black!100, inner sep=0pt, minimum size=2pt] (ltimes3) at (5,-5) {};
\node[circle, draw, fill=black!100, inner sep=0pt, minimum size=2pt] (ltimes4) at (5,-7) {};
\draw (ltimes0) -- (ltimes1) -- (ltimes2) -- (ltimes0);
\draw (ltimes4) -- (ltimes2) -- (ltimes3);
\node (ltimes) at (3.5,-8) {\small $\ltimes$};

\node[circle, draw, fill=black!100, inner sep=0pt, minimum size=2pt] (K330) at (-14,4) {};
\node[circle, draw, fill=black!100, inner sep=0pt, minimum size=2pt] (K331) at (-12.5,4) {};
\node[circle, draw, fill=black!100, inner sep=0pt, minimum size=2pt] (K332) at (-11,4) {};
\node[circle, draw, fill=black!100, inner sep=0pt, minimum size=2pt] (K333) at (-14,6) {};
\node[circle, draw, fill=black!100, inner sep=0pt, minimum size=2pt] (K334) at (-12.5,6) {};
\node[circle, draw, fill=black!100, inner sep=0pt, minimum size=2pt] (K335) at (-11,6) {};
\draw (K330) -- (K333) -- (K331) -- (K334) -- (K332) -- (K335) -- (K330);
\draw (K330) -- (K334);
\draw (K331) -- (K335);
\draw (K332) -- (K333);
\node (K33) at (-12.5,3) {\tiny $K_{3,3}$};

\node[circle, draw, fill=black!100, inner sep=0pt, minimum size=2pt] (W50) at (-12,-5) {};
\node[circle, draw, fill=black!100, inner sep=0pt, minimum size=2pt] (W51) at (-13,-4) {};
\node[circle, draw, fill=black!100, inner sep=0pt, minimum size=2pt] (W52) at (-11,-4) {};
\node[circle, draw, fill=black!100, inner sep=0pt, minimum size=2pt] (W53) at (-11,-6) {};
\node[circle, draw, fill=black!100, inner sep=0pt, minimum size=2pt] (W54) at (-13,-6) {};
\draw (W50) -- (W51);
\draw (W50) -- (W52);
\draw (W50) -- (W53);
\draw (W50) -- (W54);
\draw (W51) -- (W52) -- (W53) -- (W54) -- (W51);
\node (W5) at (-12,-7) {\tiny $W_5$};

\node[circle, draw, fill=black!100, inner sep=0pt, minimum size=2pt] (P3jP30) at (12,1.5) {};
\node[circle, draw, fill=black!100, inner sep=0pt, minimum size=2pt] (P3jP31) at (12,0) {};
\node[circle, draw, fill=black!100, inner sep=0pt, minimum size=2pt] (P3jP32) at (12,-1.5) {};
\node[circle, draw, fill=black!100, inner sep=0pt, minimum size=2pt] (P3jP33) at (14,1.5) {};
\node[circle, draw, fill=black!100, inner sep=0pt, minimum size=2pt] (P3jP34) at (14,0) {};
\node[circle, draw, fill=black!100, inner sep=0pt, minimum size=2pt] (P3jP35) at (14,-1.5) {};
\draw (P3jP30) -- (P3jP31) -- (P3jP32) -- (P3jP35) -- (P3jP34) -- (P3jP33) -- (P3jP30);
\draw (P3jP30) -- (P3jP34) -- (P3jP32);
\draw (P3jP33) -- (P3jP31) -- (P3jP35);
\draw (P3jP31) -- (P3jP34);
\draw (P3jP30) -- (P3jP35);
\draw (P3jP32) -- (P3jP33);
\node (P3jP3) at (13,-2.5) {\tiny $P_3 \vee P_3$};
\end{tikzpicture}
\caption{The sets of minimal forbidden induced subgraphs for the properties $\cb(G)\leq 2$ (A) and $\mr(G,\mathbb{F}_2)\leq 2$ (B).}
\label{im:forbind2}
\end{figure}

\begin{thm}\label{thm:forbind2}
The class of graphs 
$$\{ G \mid \cb(G)\leq 2 \}$$
is the class of $\mathcal{F}$-free graphs, where $\mathcal{F}$ is the set of graphs shown in Figure~\ref{im:forbind2} (A).
\end{thm}

\begin{proof}
Suppose, for the sake of contradiction, that there exists a graph $G = (V, E)$ such that $\cb(G) > 2$, and $G$ does not contain any element of $\mathcal{F}$ as an induced subgraph. Furthermore, suppose that $G$ is minimal with these qualities; that is, every proper induced subgraph $H$ of $G$ has $\cb(G) \leq 2$. Then $G$ has no isolated vertices. Furthermore, $|G| \geq 5$ by Theorems~\ref{thm:easyvxupp} and~\ref{thm:vxuppbound}.
\par

The rest of the proof is outlined as follows. We show that there exists a vertex $x$ for which $\cb(G-x) = 2$. Letting $\mathscr{C} = \{C_1, C_2\}$ be a minimum subgraph complementation system for $G-x$, depicted in Figure~\ref{im:forbind2diagram}, we then show that $C_1 \cap C_2$ is nonempty, and that $G-x$ has no isolated vertices. Finally, we split into two cases: either one of the sets in $\mathscr{C}$ contains the other, or not. Contradictions are derived by showing that either $\cb(G) \leq 2$, or that $G$ contains an induced subgraph in $\mathcal{F}$.
\par

Firstly, there exists a vertex $x$ for which $\cb(G-x) = 2$.  We have $\cb(G-v) \geq \cb(G)-2 \geq 1$ for all $v\in V$, since we can add $N(v)$ and $N[v]$ to any minimum subgraph complementation system for $G-v$ to obtain one for $G$. Furthermore, if $\cb(G-v)=1$ for all $v\in V$, then $\mr(G-v,\mathbb{F}_2) = 1$ for all $v\in V$, so $G$ is a minimal forbidden induced subgraph for the property $\mr(G,\mathbb{F}_2) \leq 1$. These are the graphs $P_3$ and $2K_2$, which both have subgraph complementation systems of cardinality 2, so there exists a vertex $x\in V$ such that $\cb(G-x) = 2$.
\par

Let $\mathscr{C} = \{C_1, C_2\}$ be a minimum subgraph complementation system for $G-x$. Notice that both $|C_1| \geq 2$ and $|C_2| \geq 2$. We begin by showing that $C_1 \cap C_2$ is nonempty. Suppose $C_1\cap C_2 = \emptyset$. The isolated vertices of $G-x$ are a subset of $N_G(x)$. If every neighbor of $x$ is isolated in $G-x$, then $G$ has an induced $3K_2$. Thus, $x$ has a neighbor in at least one of $C_1$ and $C_2$. Without loss of generality, say $x$ has a neighbor in $C_1$. Then $x$ dominates $C_1$, otherwise $G$ has an induced $P_3 + K_2$ (if $x$ has no neighbor in $C_2$), or an induced $P_4$ (otherwise). If $x$ has no neighbor in $C_2$, then either $\cb(G) \leq 2$, or $G$ has an induced $P_3 + K_2$. In fact, $x$ dominates $C_2$, otherwise $G$ has an induced $P_4$. Then either $\cb(G) \leq 2$, or $G$ has an induced $\ltimes$, a contradiction. Therefore, $C_1\cap C_2$ is nonempty.
\par

Suppose there exists an isolated vertex in $G-x$. Then, for each edge $uv$ of $G-x$, either both or neither of $u$ and $v$ are neighbors of $x$, otherwise $G$ has an induced $P_4$. If there are at least two isolated vertices, then for each edge $uv$ of $G-x$, exactly one of $u$ and $v$ is a neighbor of $x$, otherwise $G$ has an induced $P_3+K_2$ or an induced $\ltimes$. We conclude there is exactly one isolated vertex in $G-x$. If $x$ has no other neighbor, then $G$ has an induced $P_3+K_2$, since $C_1$ and $C_2$ are not disjoint. Without loss of generality, say $x$ has a neighbor in $C_1$. In fact, we can conclude that $x$ dominates $C_1$, otherwise $G$ has an induced $P_4$. Then $x$ has a neighbor in $C_2$, so $x$ dominates $C_2$ as well, and $G$ has an induced $\text{dart}$. Therefore, $G-x$ has no isolated vertices.
\par

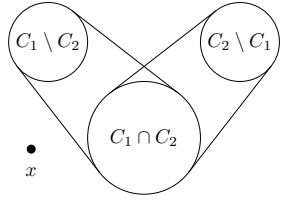
\begin{figure}
 \centering
 \scalebox{.75}{
 \begin{tikzpicture}[inner sep=0,outer sep=0]
  \pgfmathsetmacro{\rI}{10mm}
  \pgfmathsetmacro{\rC}{7mm}
  \pgfmathsetmacro{\rCC}{7mm}
  \node[draw,circle,minimum size=2*\rI pt] (I) {$C_1\cap C_2$};
  \node[draw,circle,xshift=-\rI-\rC,yshift=\rI+\rC,minimum size=2*\rC pt] (C) {$C_1\setminus C_2$};
  \node[draw,circle,xshift=\rI+\rC,yshift=\rI+\rC,minimum size=2*\rC pt] (CC) {$C_2\setminus C_1$};
  \draw[fill=black!100,black!100] (-2,-.2) circle (2pt); \draw (-2,-.6) node{$x$};
  \coordinate (c) at (barycentric cs:I=-\rC,C=\rI);
  \fill[red](c) circle (.2pt);
  \coordinate (d) at (barycentric cs:I=-\rCC,CC=\rI);
  \fill[red](d) circle (.2pt);
  \draw (tangent cs:node=C,point={(c)},solution=2) -- (tangent cs:node=I,point={(c)},solution=2);
  \draw (tangent cs:node=C,point={(c)},solution=1) -- (tangent cs:node=I,point={(c)},solution=1);
  \draw (tangent cs:node=CC,point={(d)},solution=2) -- (tangent cs:node=I,point={(d)},solution=2);
  \draw (tangent cs:node=CC,point={(d)},solution=1) -- (tangent cs:node=I,point={(d)},solution=1);
\end{tikzpicture}
 }
 \caption{A subgraph complementation system for $G-x$}
 \label{im:forbind2diagram}
\end{figure}

Figure~\ref{im:forbind2diagram} represents a minimum subgraph complementation system $\mathscr{C} = \{C_1,C_2\}$ of $G-x$. Without loss of generality, we assume that $|C_1| \leq |C_2|$. One may imagine $G-x$ as disjoint cliques $C_1\setminus C_2$ and $C_2\setminus C_1$, and an independent dominating set $C_1\cap C_2$. We now split into cases: either $C_1 \setminus C_2$ and $C_2 \setminus C_1$ are both nonempty, or $C_1 \subset C_2$. The former case is divided into subcases differentiating between the possible neighborhoods of $x$ in $G$. 
\\[1em]
\underline{Case 1:} Suppose that $C_1 \setminus C_2$ and $C_2 \setminus C_1$ are both nonempty. Throughout this section, vertices in $C_1\setminus C_2$ are be denoted by $u=u_0, u_1, u_2, \ldots$, vertices in $C_1\cap C_2$ by $w=w_0, w_1, w_2, \ldots$, and vertices in $C_2\setminus C_1$ by $z=z_0, z_1, z_2, \ldots$.
\par

Suppose $N(x) \subseteq C_1 \setminus C_2$. Then $G$ has an induced $P_4$ on vertex set $\{x, u, w, z\}$, where $u\in N(x)$, $w\in C_1\cap C_2$, and $z\in C_2 \setminus C_1$. A similar contradiction is derived if $N(x) \subseteq C_2\setminus C_1$.
\par

Suppose $N(x) \subseteq C_1\cap C_2$, and let $w = w_0 \in N(x)$. If $|C_2\setminus C_1|\geq 2$, say $z_0, z_1 \in C_2\setminus C_1$, then $G$ contains an induced $\ltimes$ on vertex set $\{x, z_0, z_1, w, u\}$, where $u\in C_1\setminus C_2$. Otherwise, since $|C_2|\geq |C_1|$ by assumption, $|C_1\setminus C_2| = |C_2\setminus C_1| = 1$. Let $C_1\setminus C_2 = \{u\}$, and let $C_2\setminus C_1 = \{z\}$. Since $|G|\geq 5$, we have $|C_1\cap C_2| \geq 2$. If $x$ has a non-neighbor in $C_1 \cap C_2$, say $w_1$, then $G$ has an induced $P_4$ on $\{x,w_0,u,w_1\}$. Otherwise, $N(x) = C_1\cap C_2$. If $C_1 \cap C_2 = \{w_0, w_1\}$, then there is a subgraph complementation system of $G$ of cardinality 2: $\{\{x,w_0,u,z\}, \{x,w_1,u,z\}\}$. Thus, there exist vertices $w_0, w_1, w_2 \in N(x) \cap (C_1\cap C_2)$, and $G$ has an induced $K_{3,3}$ on $\{x,u,z,w_0,w_1,w_2\}$.
\par

Suppose $x$ has neighbors $u = u_0 \in C_1\setminus C_2$ and $w = w_0 \in C_1\cap C_2$, but no neighbor in $C_2\setminus C_1$. Let $z \in C_2 \setminus C_1$. If $x$ has a non-neighbor $u_1 \in C_1 \setminus C_2$, then $G$ has an induced $\text{dart}$ on $\{x, u, u_1, w, z\}$, and if $x$ has a non-neighbor $w_1 \in C_1\cap C_2$, then $G$ has an induced $P_4$ on $\{x, u, w_1, z\}$. Thus, $C_1 \setminus C_2 \subset N(x)$, and $C_1 \cap C_2 \subset N(x)$. Since $G-x$ has no isolated vertices, we have $N(x) = C_1$. But then $G$ has a subgraph complementation system of cardinality 2: $\{C_1\cup \{x\}, C_2\}$. Thus, we arrive at a contradiction when $x$ has neighbors in $C_1\setminus C_2$ and $C_1 \cap C_2$ but not $C_2\setminus C_1$. By similar arguments, we derive a contradiction if $x$ has neighbors in $C_2\setminus C_1$ and $C_1\cap C_2$ but none in $C_1\setminus C_2$.
\par

Finally, suppose $x$ has neighbors $u = u_0 \in C_1\setminus C_2$, $w = w_0 \in C_1\cap C_2$, and $z = z_0 \in C_2\setminus C_1$. Since $|G|\geq 5$ and $|C_1| \leq |C_2|$, either $|C_1 \cap C_2| \geq 2$ or $|C_2 \setminus C_1| \geq 2$. Suppose $|C_2 \setminus C_1| \geq 2$. If $x$ has a non-neighbor $z_1 \in C_2 \setminus C_1$, then $G$ has an induced $P_4$ on $\{u,x,z_0,z_1\}$. Otherwise, $x$ dominates $C_2 \setminus C_1$, and $G$ has induced $\text{full house}$ on $\{u, x, w, z_0, z_1\}$. Thus, $|C_2 \setminus C_1| = |C_1 \setminus C_2| = 1$, and $|C_1\cap C_2|\geq 2$. If $x$ has 2 or more neighbors in $C_1\cap C_2$, say $w_0,w_1\in N(x) \cap C_1\cap C_2$, then $G$ has an induced $W_5$ on $\{x,u,w_0,w_1,z\}$. Thus, $x$ has a non-neighbor $w_1$ in $C_1 \cap C_2$. Suppose $C_1\cap C_2 = \{w_0,w_1\}$. Since $C_1\setminus C_2 = \{u\}$ and $C_2\setminus C_1 = \{z\}$, the two sets $\{x,u,w_0,z\}$ and $\{w_1,u,z\}$ form a subgraph complementation system of $G$. Now suppose that $|C_1\cap C_2|\geq 3$; say $w_0, w_1, w_2 \in C_1\cap C_2$. We have seen that $w_0$ is the only neighbor of $x$ in $C_1 \cap C_2$. Thus, $G$ contains an induced $\ltimes$ on $\{x,u,w_0,w_1,w_2\}$. We conclude that $x$ must not have neighbors in each of $C_1\setminus C_2$, $C_1\cap C_2$, and $C_2\setminus C_1$. This concludes Case 1.
\\[1em]
\underline{Case 2:} Suppose that $C_1 \setminus C_2$ is empty, {\em i.e.} $C_1\subset C_2$.
\par

Let $u_0, u_1\in C_1$ and $z = z_0 \in C_2\setminus C_1$. If $N(x)\subsetneq C_1$, say $u_0\in N(x)$ and $u_1 \in C_1\setminus N(x)$, and if $z \in C_2\setminus C_1$, then $G$ has an induced $P_4$ on $\{x, u_0, z, u_1\}$. If $N(x) = C_1$, then $G$ has a subgraph complementation system of cardinality 2: $\{C_2, C_1\cup \{x\}\}$. Thus, $x$ has a neighbor $z\in C_2 \setminus C_1$. If $u_0\in N(x)$ but $u_1, u_2 \in C_1$ are not neighbors of $x$, then $G$ has an induced $\ltimes$ on $\{x,u_0,u_1,u_2,z\}$. If $x$ has neighbors $u_0, u_1 \in C_1$, and a non-neighbor $u_2 \in C_1$, then $G$ has an induced $\text{dart}$ on $\{x,u_0,u_1,u_2,z\}$. Thus, $x$ dominates $C_1$. If $x$ also dominates $C_2$, then $G$ has a subgraph complementation system of cardinality 2: $\{C_1, C_2\cup \{x\}\}$. Thus, $x$ has a neighbor $z_0$ and a non-neighbor $z_1$ in $C_2 \setminus C_1$, and $G$ has an induced $W_5$ on $\{x,u_0,u_1,z_0,z_1\}$. This completes the proof. 
\end{proof}

\section*{Acknowledgements}

The authors would like to thank the anonymous referee for suggesting Theorem~\ref{thm:mrcbt2}. We would also like to thank Alexander Clifton, Eric Culver, Jiaxi Nie, Jason O'Neill, and Mei Yin for helpful discussions about symmetric differences of complete tripartite graphs following the 2021 Graduate Research Workshop in Combinatorics.

\bibliographystyle{plain}
\bibliography{references.bib}

\end{document}